\documentclass[12pt,twoside]{amsart}
\usepackage{amsmath, amsthm, amscd, amsfonts, amssymb, graphicx}
\usepackage{enumerate}
\usepackage[colorlinks=true,
linkcolor=blue,
urlcolor=cyan,
citecolor=red]{hyperref}
\usepackage{mathrsfs}
\newtheorem{theorem}{Theorem}[section]
\newtheorem{lemma}{Lemma}[section]
\newtheorem{remark}{Remark}[section]
\newtheorem{definition}{Definition}[section]
\newtheorem{corollary}{Corollary}[section]
\newtheorem{example}{Example}[section]
\newtheorem{proposition}{Proposition}[section]
\numberwithin{equation}{section}

\begin{document}
	
\title{A new treatment of convex functions}
\author{Mohammad Sababheh, Shigeru Furuichi, and Hamid Reza Moradi}
\maketitle

\begin{abstract}
Convex functions have played a major role in the field of Mathematical inequalities. In this paper, we introduce a new concept related to convexity, which proves better estimates when the function is somehow more convex than another.\\
In particular, we define what we called $g-$convexity as a generalization of $\log-$convexity. Then we prove that $g-$convex functions have better estimates in  certain known inequalities like the Hermite-Hadard inequality, super additivity of convex functions, the Majorization inequality and some means inequalities.\\
Strongly related to this, we define the index of convexity as a measure of ``how much the function is convex".\\
Applications including Hilbert space operators, matrices and entropies will be presented in the end.
\end{abstract}
\pagestyle{myheadings}
\markboth{\centerline {M. Sababheh, S. Furuichi \& H. R. Moradi}}
{\centerline {A New Treatment of Convex Functions}}
\bigskip
\bigskip

\section{Introduction}
A function $f:[a,b]\to\mathbb{R}$ is said to be convex if
$$f(w_1x_1+w_2x_2)\leq w_1f(x_1)+w_2f(x_2),$$ for all $x_1,x_2\in [a,b]$ and positive numbers $w_1,w_2$ satisfying $w_1+w_2=1.$ This is generalized by the so called Jensen's inequality in the form
\begin{equation}\label{jensen_intro}
f\left(\sum_{i=1}^{n}w_ix_i\right)\leq \sum_{i=1}^{n}w_if(x_i), n\in\mathbb{N},
\end{equation}
for $x_i\in [a,b]$ and $w_i>0$ with $\sum_{i=1}^{n}w_i=1.$

Convex functions have received a considerable attention in the literature due to their applications in many scientific fields, such as Mathematical inequalities, Mathematical analysis and Mathematical physics.\\
It can be seen that all known properties of convex functions follow from \eqref{jensen_intro}. Very recently, a new characterization of convex functions was given in \cite{sab_bull}, where nonlinear upper bounds of convex functions were found. In this context, we recall that the geometric meaning of a convex function is that the function is bounded above by its linear secants.

However, neither the original definition nor the Jensen inequality differentiates between two convex functions. In other words, when $f_1$ and $f_2$ are convex functions, all what the definition says is that
$$f_i(w_1x_1+w_2x_2)\leq w_1f_i(x_1)+w_2f_i(x_2).$$ This does not reflect any of the many other properties of $f_i$. For example, if $f_1(x)=x^2$ and $f_2(x)=x^4$, then both functions are convex. Hence,
$$(w_1x_1+w_2x_2)^2\leq w_1x_1^2+w_2x^2\;{\text{and}}\;(w_1x_1+w_2x_2)^4\leq w_1x_1^4+w_2x^4.$$

The main goal of this article is somehow to look into ``how much the convex function is convex?" For example, according to our argument, we will see that $f(x)=x^4$ is ``more convex" than $f(x)=x^2,$ and then to see that $f(x)=e^x$ is more convex than polynomials!\\
The idea we present is a simple idea, where we make a concave function operates on the convex function, then to see the result. For example, the function $f(x)=x^2, x>0$ is convex. It is somehow about ``how much power do we need to exert to stop convexity of $f$?" In this case, we know that $\sqrt{f(x)}=x.$ The function $x$ being the ``least" convex function, we see that we needed a power of $\frac{1}{2}$ to stop convexity of $f(x)=x^2$, somehow.

Our main target is to formalize the above paragraph! We will see that our approach generalizes the well known and useful notion of $\log-$convexity, where a function $f$ is called $\log-$convex if the function $\log f$ is convex. It is well known that $\log-$convex functions satisfy better bounds than convex functions. We notice here that the function $g(x)=\log x$ is a concave function that acted on $f$. Having $\log f$ convex made $\log-$convex functions satisfy better results than convex functions. 

Our main definition reads as follows.

\begin{definition}
Let $f:J_1\to J_2$ be a continuous  function on the interval $J_1$ and let $g:J_2\to J_3$ be increasing and concave (resp., convex) on $J_2$, such that $g\circ f:J_1\to J_3$  is convex (resp., concave). Then, $f$ is said to be $g-$convex (resp., $g-$concave).
\end{definition}

We observe that, in this definition, we do not impose the condition that $f$ is convex. However, this follows immediately because 
$$f=g^{-1}\left(g\circ f\right);$$ which is convex since $g\circ f$ is convex and $g^{-1}$ is convex and increasing.\\
We will show that $g-$convex functions satisfy better bounds than convex functions. However, the significance here is that we treat convex functions as $g-$convex functions, for certain $g$. Once this idea is established, we show Jensen-type  and Hermite-Hadamard inequalities, as refinements of the well known inequalities.

As a special case, we will take the power functions $g(x)=x^{\frac{1}{r}}, r\geq 1,$ to introduce the new notion of ``the index of a convex function". This new convexity index aims to present a number that, somehow, measures convexity of $f$. As a consequence of this index, we will be able to present a new property of convex functions. Namely, we will show that a positive convex function $f$ satisfies 
$$(f')^2\leq ff''\; {\text{if and only if the index of convexity of}}\;f\;{\text{is}}\;\infty,$$
as a new property of convex functions relating $f,f'$ and $f''.$

Then we present some applications for Hilbert space operators and entropies. These applications include better majorization bounds, better bounds in the operator-convex super additivity results and the Jensen inner product inequality.

\section{treatment of convex inequalities}
In this section, we present some applications of $g-$convex functions in the context of the Jensen inequality, the Hermite-Hadamard inequality and some applications to mean inequalities. Also, super additivity of convex functions will be visited.

\begin{proposition}\label{7}
Let $f$ be a $g$-convex function on the interval $J$. Then $f$ is convex and 
\begin{equation*}
f\left( \sum\limits_{i=1}^{n}{{{w}_{i}}{{x}_{i}}} \right)\le {{g}^{-1}}\left( \sum\limits_{i=1}^{n}{{{w}_{i}}\left( g\circ f \right)\left( {{x}_{i}} \right)} \right)\le \sum\limits_{i=1}^{n}{{{w}_{i}}f\left( {{x}_{i}} \right)}
\end{equation*}
for any ${{x}_{1}},\ldots ,{{x}_{n}}\in J$ and $0\le {{w}_{1}},\ldots ,{{w}_{n}}\le 1$ with $\sum\nolimits_{i=1}^{n}{{{w}_{i}}}=1$.
\end{proposition}
\begin{proof}
Since $g\circ f$ is a convex function, we have for any ${{x}_{1}},\ldots ,{{x}_{n}}\in J$ and $0\le {{w}_{1}},\ldots ,{{w}_{n}}\le 1$ with $\sum\nolimits_{i=1}^{n}{{{w}_{i}}}=1$,
\[\left( g\circ f \right)\left( \sum\limits_{i=1}^{n}{{{w}_{i}}{{x}_{i}}} \right)\le \sum\limits_{i=1}^{n}{{{w}_{i}}\left( g\circ f \right)\left( {{x}_{i}} \right)}.\]
Since $g^{-1}$ is increasing and convex, we then have
\[\begin{aligned}
   f\left( \sum\limits_{i=1}^{n}{{{w}_{i}}{{x}_{i}}} \right)&={{g}^{-1}}\left( \left( g\circ f \right)\left( \sum\limits_{i=1}^{n}{{{w}_{i}}{{x}_{i}}} \right) \right) \\ 
 & \le {{g}^{-1}}\left( \sum\limits_{i=1}^{n}{{{w}_{i}}\left( g\circ f \right)\left( {{x}_{i}} \right)} \right) \\ 
 & \le \sum\limits_{i=1}^{n}{{{w}_{i}}{{g}^{-1}}\left( \left( g\circ f \right)\left( {{x}_{i}} \right) \right)} \\ 
 & =\sum\limits_{i=1}^{n}{{{w}_{i}}f\left( {{x}_{i}} \right)}.  
\end{aligned}\]
This, in particular, shows that $f$ is convex. Consequently,
\begin{equation}\label{1}
f\left( \sum\limits_{i=1}^{n}{{{w}_{i}}{{x}_{i}}} \right)\le {{g}^{-1}}\left( \sum\limits_{i=1}^{n}{{{w}_{i}}\left( g\circ f \right)\left( {{x}_{i}} \right)} \right)\le \sum\limits_{i=1}^{n}{{{w}_{i}}f\left( {{x}_{i}} \right)}.
\end{equation}
\end{proof}

Clarify $g-$convexity, we present some examples.
\begin{example}
\hfill
\begin{itemize}
\item[(i)] If we take $f(x):=\exp(x)$ and $g(x):=\log x$, $(x>0)$, then $h(x):=(g\circ f)(x)=x$ and we have
$$
\exp\left(\sum_{i=1}^nw_ix_i\right)\leq \sum_{i=1}^n w_i\exp(x_i).
$$
The inequality is just Jensen's inequality. If we take $\log$ for this inequality, we get
$$
\sum_{i=1}^n w_ix_i \leq \log\left(\sum_{i=1}^nw_i\exp(x_i)\right).
$$

\item[(ii)] If we take $f(x):=-\log x$, $(0<x \leq 1)$ and $g(x):=x^p$, $(x>0,\,\,0\leq p \leq 1)$, then $h(x)=(-\log x)^p$ and we have
$$
-\log\left(\sum_{i=1}^n w_ix_i\right)\leq \left(\sum_{i=1}^n w_i(-\log x_i)^p\right)^{1/p}\leq -\sum_{i=1}^n w_i\log x_i,
$$
which implies
$$
\log \prod_{i=1}^n x_i^{w_i}\leq \log \exp\left(-\sum_{i=1}^nw_i(-\log x_i)^p\right)^{1/p}\leq \log \left(\sum_{i=1}^n w_ix_i \right).
$$
If we take $p=1$, then we get
\begin{equation*}
 \prod_{i=1}^n x_i^{w_i}\le  \sum_{i=1}^n w_ix_i.  
\end{equation*}
\item[(iii)] If we take $f(x):=\exp(x)$ and $g(x):=x^p$, $(0\leq p \leq 1)$, then $(g\circ f)(x)=\exp(px)$, and we have
$$
\exp\left(\sum_{i=1}^n w_ix_i\right) \leq \left(\sum_{i=1}^n w_i\exp(p x_i)\right)^{1/p}\leq \sum_{i=1}^n w_i\exp(x_i),
$$
which improves the inequality given in (i).
\item[(iv)] If we take $f(x):=x^p$, $(x>0,\,\,p\leq 0)$ and $g(x):=\log x$, $(x>0)$, then $(g\circ f)(x)=p\log x$, and we have
$$
\left(\sum_{i=1}^nw_ix_i\right)^p\le \prod_{i=1}^{n}x_i^{pw_i}\leq \sum_{i=1}^n w_ix_i^p.
$$
\end{itemize}
\end{example}

We make some space in the following example for the celebrated Young's inequality. Recall that if $a,b>0$ and $0\leq t\leq 1,$ then Young's inequality states that
\begin{equation}\label{young_ineq}
a^{1-t}b^{t}\leq (1-t)a+tb.
\end{equation}
This inequality has attracted numerous researchers due to its applications in operator theory and functional analysis, in general.
In the following, we present refinements of this inequality using our idea about $g-$convexity.
\begin{proposition}
Let $a,b>0$ and $0\leq t\leq 1.$
\begin{itemize}
\item If $0\leq p\leq 1$, then \eqref{young_ineq} can be refined as
\begin{equation}\label{ref_young_1}
a^{1-t}b^{t}\leq \left\{(1-t)a^p+tb^p\right\}^{\frac{1}{p}}\leq (1-t)a+tb.
\end{equation}
\item We also have for $0\leq p\leq 1$,
\begin{equation}\label{Heinz_1}
\sqrt{ab} \leq H^{1/p}_t(a^p,b^p) \leq H_t(a,b),
\end{equation}
where $H_t(a,b):=\dfrac{a^{1-t}b^t+a^{t}b^{1-t}}{2}$ is the Heinz mean.
\end{itemize}
\end{proposition}
\begin{proof}
Let $f(t)=a^{1-t}b^{t}$ and $g(t)=t^{p}, (0\leq p\leq 1).$ Then, $g$ is increasing concave and $g\circ f = a^{p(1-t)}b^{pt}$ is convex, since we have $(g\circ f)''(t) = a^{p(1-t)}b^{pt}p^2\left(\log a -\log b\right)^2\geq 0$. Applying Proposition \ref{7}, with $n=2$,$w_1=t,x_1=1,w_2=1-t$ and $x_2=0$ implies \eqref{ref_young_1}.

In the similar setting such as  $f(t)=a^{1-t}b^{t}$ and $g(t)=t^{p}, (0\leq p\leq 1)$ with $n=2$, $w_1=w_2=\dfrac{1}{2}$, $x_1=t$ and $x_2=1-t$ in  Proposition \ref{7}, we have \eqref{Heinz_1}.
\end{proof}
Notice that the inequality \eqref{ref_young_1} is the well known power mean inequality. Thus, we have obtained this celebrated inequality as a special case of our general argument. We note that
$
\lim\limits_{p\to 0}H^{1/p}_t(a^p,b^p) =\sqrt{ab}
$
and $H^{1/p}_t(a^p,b^p)=H_t(a,b)$ when $p=1$.
\begin{remark}
In the process, we used (i) the convexity of  $g \circ f$, and (ii) the convexity of $g^{-1}$ which is equivalent to the concavity of $g$. Note that we do not impose the condition on $f$ itself.

In addition, we can obtain the following inequalities:
$$
(g\circ f)((1-v)a+vb)\leq (1-v)(g\circ f)(a)+v (g\circ f)(b)\leq g((1-v)f(a)+vf(b))
$$
for convex $g\circ f$ and concave $g$. Also we have
$$
(g\circ f)((1-v)a+vb)\geq (1-v)(g\circ f)(a)+v (g\circ f)(b)\geq g((1-v)f(a)+vf(b))
$$
for concave $g\circ f$ and convex $g$.  
\end{remark}

On the other hand, $g-$convex functions satisfy better super additivity inequalities. Recall that a convex function $f:[0,a]\to\mathbb{R}$ with $f(0)\leq 0,$ satisfies $$f(x)+f(y)\leq f(x)+f(y), x,y\in [0,a].$$ The following result presents a better bound for $g-$convex functions.

\begin{proposition}
Let $f$ be a $g$-convex function on the interval $J:=[0,a], a>0,$ with $\left( g\circ f \right)\left( 0 \right)\le 0$ and $g(0)\ge 0$. Then
\begin{equation*}
f\left( x \right)+f\left( y \right)\le {{g}^{-1}}\left( \left( g\circ f \right)\left( x \right)+\left( g\circ f \right)\left( y \right) \right)\le f\left( x+y \right),
\end{equation*}
for any $x,y\in J$.
\end{proposition}
\begin{proof}
Since $h=g\circ f$ is a convex function with $\left( g\circ f \right)\left( 0 \right)\le 0$, we have for any $x,y\in J$,
 \[\left( g\circ f \right)\left( x \right)+\left( g\circ f \right)\left( y \right)\le \left( g\circ f \right)\left( x+y \right).\]
Since $g^{-1}$ is increasing and convex with $g\left( 0 \right)\geq 0$, we have $g^{-1}(0) \le 0$ and then have
\[\begin{aligned}
   f\left( x \right)+f\left( y \right)&={{g}^{-1}}\left( \left( g\circ f \right)\left( x \right) \right)+{{g}^{-1}}\left( \left( g\circ f \right)\left( y \right) \right) \\ 
 & \le {{g}^{-1}}\left( \left( g\circ f \right)\left( x \right)+\left( g\circ f \right)\left( y \right) \right) \\ 
 & \le {{g}^{-1}}\left( \left( g\circ f \right)\left( x+y \right) \right) \\ 
 & =f\left( x+y \right).  
\end{aligned}\]
This completes the proof of the proposition. 
\end{proof}

Our next target is improving the Hermite-Hadamard inequality for $g-$convex functions. We observe that $g-$convex functions satisfy better bounds in the Hermite-Hadamrd inequality than mere convex functions.

\begin{theorem}
Let $f$ be a $g-$ convex function on the interval $J$.  Then for $a<b$ in $J$,
\[\begin{aligned}
   f\left( \frac{a+b}{2} \right)&\le \frac{1}{b-a}\int_{a}^{b}{f\left( z \right)dz} \\ 
 & \le \int_{a}^{b}g^{-1}\left( \frac{z-a}{b-a}h\left( a \right)+\frac{b-z}{b-a}h\left( b \right) \right)dz \\ 
 & \le \frac{f\left( a \right)+f\left( b \right)}{2},  
\end{aligned}\]
where $h=g\circ f.$
\end{theorem}
\begin{proof}
On account of Proposition \ref{7}, it follows that
\begin{equation}\label{2}
f\left( \left( 1-v \right)x+vy \right)\le g^{-1}\left( \left( 1-v \right)(g\circ f)\left( x \right)+v(g\circ f)\left( y \right) \right)\le \left( 1-v \right)f\left( x \right)+vf\left( y \right).
\end{equation}

Now, suppose  $z\in \left[ a,b \right]$. If we substitute $x=a$, $y=b$, and $1-v={\left( b-z \right)}/{\left( b-a \right)}\;$ in \eqref{2}, we get
\begin{equation}\label{3}
\begin{aligned}
   f\left( z \right)&\le g^{-1}\left( \frac{b-z}{b-a}h\left( a \right)+\frac{x-a}{b-a}h\left( b \right) \right) \\ 
 & \le \frac{b-z}{b-a}f\left( a \right)+\frac{z-a}{b-a}f\left( b \right).  
\end{aligned}
\end{equation}
Since $z\in \left[ a,b \right]$, it follows that $b+a-z\in \left[ a,b \right]$. Now, applying the inequality \eqref{3} to the variable $b+a-z$, we get
\begin{equation}\label{4}
\begin{aligned}
   f\left( b+a-z \right)&\le g^{-1}\left( \frac{z-a}{b-a}h\left( a \right)+\frac{b-z}{b-a}h\left( b \right) \right) \\ 
 & \le \frac{z-a}{b-a}f\left( a \right)+\frac{b-z}{b-a}f\left( b \right).  
\end{aligned}
\end{equation}
By adding inequalities \eqref{3} and \eqref{4}, we infer that
\[\begin{aligned}
  & f\left( b+a-z \right)+f\left( z \right) \\ 
 & \le g^{-1}\left( \frac{z-a}{b-a}h\left( a \right)+\frac{b-z}{b-a}h\left( b \right) \right)+\left( \frac{b-z}{b-a}h\left( a \right)+\frac{z-a}{b-a}h\left( b \right) \right) \\ 
 & \le \frac{z-a}{b-a}f\left( a \right)+\frac{b-z}{b-a}f\left( b \right)+\frac{b-z}{b-a}f\left( a \right)+\frac{z-a}{b-a}f\left( b \right) \\ 
 & =f\left( b \right)+f\left( a \right)  
\end{aligned}\]
which, in turn, leads to
\begin{equation}\label{5}
\begin{aligned}
   f\left( \frac{a+b}{2} \right)&=f\left(\frac{a+b-z+z}{2}\right) \\ 
 & \le \frac{f\left( a+b-z \right)+f\left( z \right)}{2} \\ 
 & \le \frac{1}{2}\left(g^{-1}\left( \frac{z-a}{b-a}{h}\left( a \right)+\frac{b-z}{b-a}{h}\left( b \right) \right)+g^{-1}\left( \frac{b-z}{b-a}{h}\left( a \right)+\frac{z-a}{b-a}{h}\left( b \right) \right) \right) \\ 
 & \le \frac{f\left( a \right)+f\left( b \right)}{2}.  
\end{aligned}
\end{equation}
Now, the result follows by integrating the inequality \eqref{5} over $z\in \left[ a,b \right]$, and using the fact that $\int_{a}^{b}{f\left( z \right)dz}=\int_{a}^{b}{f\left( a+b-z \right)dz}$.
\end{proof}

With the same approach, we can provide another refinement of Hermite-Hadamard inequality.
\begin{theorem}
Let $f$ be a $g-$ convex function on the interval $J$.  Then for $a<b$ in $J$,
\[\begin{aligned}
   f\left( \frac{a+b}{2} \right)&\le \int_{0}^{1}g^{-1}\left( \frac{h\left( \left( 1-v \right)a+vb \right)+h\left( \left( 1-v \right)b+va \right)}{2} \right)dv \\ 
 & \le \int_{0}^{1}{f\left( \left( 1-v \right)a+vb \right)dv} \\ 
 & \le \frac{1}{2}\left( \int_{0}^{1}g^{-1}\left( \left( 1-v \right)h\left( a \right)+vh\left( b \right) \right)dv+\int_{0}^{1}g^{-1}\left( \left( 1-v \right)h\left( b \right)+vh\left( a \right) \right)dv \right) \\ 
 & \le \frac{f\left( a \right)+f\left( b \right)}{2}, 
\end{aligned}\]
where $h=g\circ f.$
\end{theorem}
\begin{proof}
The inequality \eqref{2} implies that
\begin{equation}\label{6}
\begin{aligned}
   f\left( \frac{a+b}{2} \right)&=f\left( \frac{\left( 1-v \right)a+vb+\left( 1-v \right)b+va}{2} \right) \\ 
 & \le g^{-1}\left( \frac {h\left( \left( 1-v \right)a+vb \right)+h\left( \left( 1-v \right)b+va \right)}{2} \right) \\ 
 & \le \frac{f\left( \left( 1-v \right)a+vb \right)+f\left( \left( 1-v \right)b+va \right)}{2} \\ 
 & \le \frac{g^{-1}\left( \left( 1-v \right)h\left( a \right)+vh\left( b \right) \right)+g^{-1}\left( \left( 1-v \right)h\left( b \right)+vh\left( a \right) \right)}{2} \\ 
 & \le \frac{\left( 1-v \right)f\left( a \right)+vf\left( b \right)+\left( 1-v \right)f\left( b \right)+vf\left( a \right)}{2} \\ 
 & =\frac{f\left( a \right)+f\left( b \right)}{2}.  
\end{aligned}
\end{equation}
Now, the result follows by integrating the inequality \eqref{6} over $v\in \left[ a,b \right]$.
\end{proof}

\section{Index of convexity}
In this section, we define the index of convexity as a positive real number that, somehow, measures how convex the functions is. According to this definition, we will see that a function with larger index of convexity is more convex. This definition is motivated by our earlier discussion of convexity of $g\circ f.$ So, if we select $g(x)=x^{\frac{1}{r}}, r\geq 1,$ we reach the following definition.
\begin{definition}
Let $f:(a,b)\to (0,\infty)$ be a convex function. With $f$, we associate a set of real numbers called {\textit{the set of convex exponents of $f$}} and defined by $$C_{exp}(f)=\{r\geq 1:(f(x))^{\frac{1}{r}}\;{\text{is convex}}\}.$$ The index of convexity of $f$ is then defined by
$$I_{conv}(f)=\sup_r C_{exp}(f).$$
\end{definition}
\begin{example}
It can be easily seen that the power function $f(x)=x^{r}, r\geq 1$ has index of convexity $I_{conv}(f)=r.$\\
On the other hand, if $f(x)=e^{x}$, then 
$$C_{exp}(f)=[1,\infty)\;{\text{and}}\;I_{conv}(f)=\infty.$$
Moreover, the function $f(x)=x^{-1}$, defined on $(0,\infty)$ satisfies
$$C_{exp}(f)=[1,\infty)\;{\text{and}}\;I_{conv}(f)=\infty.$$
Further, the function $f(x)=\tan x$ is convex on $(0,\pi/2)$, with index of convexity 1.
\end{example}

We show some properties of those newly defined concepts.

\begin{proposition}
Let $f:(a,b)\to (0,\infty)$ be a given convex function. Then $C_{exp}$ is an interval. 
\end{proposition}
\begin{proof}
We first prove that if for some $r>1,$ the function $k_r(x):=(f(x))^{\frac{1}{r}}$ is concave, then so is $k_{r'}$ for any $r'>r.$ Indeed, assuming concavity of $k_r$, we have, for $\alpha,\beta>0$ with $\alpha+\beta=1,$
\begin{align*}
k_{r'}(\alpha x+\beta y)&=(k_{r}(\alpha x+\beta y))^{\frac{r}{r'}}\\
&\geq \left( \alpha k_r(x)+\beta k_r(y)   \right)^{\frac{r}{r'}}\;({\text{by concavity of}}\;k_r)\\
&\geq \alpha (k_r(x))^{\frac{r}{r'}}+\beta (k_r(y))^{\frac{r}{r'}}\;({\text{by concavity of}}\;t\mapsto t^{\frac{r}{r'}})\\
&=\alpha k_{r'}(x)+\beta k_{r'}(y).
\end{align*}
This shows that if $r\not \in C_{exp}(f)$, then $r'\not \in C_{exp}(f)$ for all $r'>r.$\\
Now, if $I_{conv}(f)=\infty$, then $C_{exp}(f)=[1,\infty).$ If not, there would be an $r>1$ such that $r\not \in C_{exp}(f)$, which then implies $r'\not \in C_{exp}(f)$ for all $r'>r,$ which implies that $C_{exp}(f)\subseteq [1,r),$ and hence $I_{conv}(f)\leq r,$ contradicting the assumption that $I_{conv}(f)=\infty.$\\
On the other hand, if $I_{conv}(f)<\infty$, then a similar argument implies that $C_{exp}(f)=[1,I_{conv}(f)].$ 

Thus, we have shown that for any convex $f$, either $ C_{exp}(f)=[1,I_{conv}(f)]$ or $C_{exp}(f)=[1,\infty)$, which completes the proof.
\end{proof}

We know that a twice differentiable convex function satisfies $f''\geq 0.$ In fact, it turns out the the index of convexity can be used to present a new relation between $f, f'$ and $f''$ for convex functions. More precisely, we have the following.

\begin{theorem}\label{thm_new_relation}
Let $f:(a,b)\to (0,\infty)$ be a twice differentiable convex function . Then
$$I_{conv}(f)=\sup\left\{r\geq 1:\left(1-\frac{1}{r}\right)(f'(x))^2\leq f(x)f''(x), \forall x\in (a,b)\right\}.$$ In particular, $(f'(x))^2\leq f(x)f''(x)$ if and only if $I_{conv}(f)=\infty.$ 
\end{theorem}
\begin{proof}
Let $k_r(x)=(f(x))^{\frac{1}{r}}.$ Convexity of $k_r$ implies positivity of $k_r''$. Direct calculus computations then imply 
$$k_r''\geq 0 \Leftrightarrow \left(1-\frac{1}{r}\right)(f')^2\leq ff'',$$ which implies the first assertion, by definition of $I_{conv}(f).$ The second assertion follows immediately from the first.
\end{proof}
Therefore, the above theorem presents a necessary and sufficient condition for a convex function to satisfy $(f'(x))^2\leq f(x)f''(x);$ as a new property of convex functions.

At this stage, it is interesting to ask about when we can have an equality in both quantities appearing in Theorem \ref{thm_new_relation}. Namely, when do we have
$$\left(1-\frac{1}{r}\right)(f')^2= ff''\;{\text{or}}\;(f')^2= ff''.$$ This is nicely described next. Solving these two ordinary differential equations, we have.
\begin{proposition}\label{prop_equal_diff}
Let $f$ be a twice differentiable function. Then
\begin{itemize}
\item $\left(1-\frac{1}{r}\right)(f')^2= ff'', r>1,$ if and only if 
$$f(x)=\left(\frac{c}{r}x+d\right)^{r}, c,d\in\mathbb{R}.$$
\item $(f')^2= ff''$ if and only if $$f(x)=\alpha e^{\beta x}, \alpha,\beta\in\mathbb{R}.$$
\end{itemize}
\end{proposition}

In fact, simple Calculus computations lead to a full characterization of convex functions having index of convexity $I_{conv}(f)=\infty.$ This is explained in the next result.
\begin{proposition}\label{prop_func>exp}
Let $f:[a,b]\to (0,\infty)$ be an increasing convex function satisfying $$f(x)\not=0, \forall x\in [a,b]\;{\text{and}}\;(f')^2\leq ff''.$$ Then, for certain real numbers $\alpha$ and $\beta$,
$$f(x)\geq \alpha e^{\beta x}.$$
\end{proposition}
\begin{proof}
Observe first that the condition that $f$ is convex follows from the inequality $(f')^2\leq ff''.$ So, we may remove this from the statement of the proposition.\\
Now, rearranging the given inequality, we have for $x\in [a,b],$
$$\frac{f'}{f}\leq \frac{f''}{f'}\Rightarrow \int_{a}^{x}\frac{f'}{f}dt\leq \int_{a}^{x}\frac{f''}{f'}dt.$$ Performing the integrals implies
$$\log \frac{f(x)}{f(a)}\leq \log\frac{f'(x)}{f'(a)}\Rightarrow  \frac{f(x)}{f(a)}\leq \frac{f'(x)}{f'(a)}.$$ The latter inequality implies
$$\frac{f'(x)}{f(x)}\geq \frac{f'(a)}{f(a)}\Rightarrow \log \frac{f(x)}{f(a)}\geq \frac{f'(a)}{f(a)} (x-a).$$ This implies that
$$f(x)\geq f(a)\exp\left(\frac{f'(a)}{f(a)} (x-a)\right),$$ which implies the desired conclusion.
\end{proof}
Combining Theorem \ref{thm_new_relation} with Propositions \ref{prop_equal_diff} and \ref{prop_func>exp} implies the following observation.
\begin{corollary}
Let $f:[a,b]\to (0,\infty)$ be an increasing convex function. If $I_{conv}(f)=\infty$, then $f(x)\geq \alpha e^{\beta x},$ for some positive real numbers $\alpha$ and $\beta.$
\end{corollary}
At this point, it is worth looking at the function $f(x)=x^{-1}, [1,\infty).$ This function satisfies $I_{conv}(f)=\infty,$ however it is not increasing! Therefore, it does not satisfy the conclusion of the above corollary.

Next, we present the following relation between $\log-$convexity and index of convexity.
\begin{proposition}\label{prop_log_conv_ind}
Let $f$ be a $\log-$convex function on the interval $J$. Then $I_{conv}(f)=\infty.$
\end{proposition}
\begin{proof}
If $f$ is $\log-$convex, then $g(x)=\log f(x)$ is convex. Let $k_r(x)=(f(x))^{\frac{1}{r}}, r\geq 1.$ Then
$$\log k_r(x)=\frac{1}{r}\log f(x),$$ which is convex. Therefore, $k_r$ is convex for all $r\geq 1.$ This implies that $I_{conv}(f)=\infty.$
\end{proof}

\section{Applications to Hilbert space operators}

In this section we study operator inequalities for a composite function of two functions. We remind the reader, first, of some terminologies and notations. Let $\mathcal{B}(\mathcal{H})$ denote the $C^*-$algebra of all bounded linear operators acting on a Hilbert space $\mathcal{H}$. When $\mathcal{H}$ is finite dimensional, say of dimension $n$, the algebra $\mathcal{B}(\mathcal{H})$ is identified with the algebra of all complex $n\times n$ matrices, denoted $\mathcal{M}_n$. A real function $f$ defined on an interval $J$ is said to be operator monotone if $f(A)\geq f(B)$ whenever $A,B\in\mathcal{B}(\mathcal{H})$ are self adjoint operators (or Hermitian matrices) such that $A\geq B$, with spectra in $J$. In this context, we write $A\geq B$ if $A-B$ is a positive operator. That is, if $\left<Ax,x\right>\geq \left<Bx,x\right>$ for all vectors $x\in\mathcal{H}.$ On the other hand, $f$ will be called an operator convex function if for any pair of self adjoint operators $A,B\in\mathcal{B}(\mathcal{H})$ and any $t\in[0,1]$, we have the convex inequality $f((1-t)A+tB)\leq (1-t)f(A)+tf(B).$ Operator concave functions are defined similarly.

Firstly, we consider two continuous positive functions $f$ and $g$ defined on $(0,\infty)$.
If $f$ and $g$ are operator monotone functions, then the composite function $g\circ f$ is clearly operator monotone.   For slightly different conditions on $f$ and $g$, we have the following theorem.

\begin{theorem}
Let $f:\left[ 0,\infty  \right)\to \left[ 0,\infty  \right)$ be a real-valued continuous function. If $g:\left[ 0,\infty  \right)\to \left[ 0,\infty  \right)$ is increasing operator convex such that $g\circ f$ is operator concave, then $f$ is  operator concave. In particular, if $A,B\in \mathcal{B}\left( \mathcal{H} \right)$ are two positive operators then
	\[\begin{aligned}
   f\left( \left( 1-v \right)A+vB \right)&\ge {{g}^{-1}}\left( \left( 1-v \right)(g\circ f)\left( A \right)+v(g\circ f)\left( B \right) \right) \\ 
 & \ge \left( 1-v \right)f\left( A \right)+vf\left( B \right). 
\end{aligned}\]
\end{theorem}
\begin{proof}
It follows from the operator concavity of $g\circ f$ that
	\[(g\circ f)\left( \left( 1-v \right)A+vB \right)\ge \left( 1-v \right)(g\circ f)\left( A \right)+v(g\circ f)\left( B \right).\]
On the other hand, it is shown in \cite[Proposition 2.3]{2} that if $g$ is an increasing operator convex function on $\left[ 0,\infty  \right)$, then ${{g}^{-1}}$ is operator monotone on $\left[ 0,\infty  \right)$. Thus,
	\[\begin{aligned}
   f\left( \left( 1-v \right)A+vB \right)&\ge {{g}^{-1}}\left( \left( 1-v \right)(g\circ f)\left( A \right)+v(g\circ f)\left( B \right) \right) \\ 
 & \ge \left( 1-v \right)f\left( A \right)+vf\left( B \right),  
\end{aligned}\]
where the second inequality follows from the fact that a function $h$ is operator monotone on a half-line $\left[ 0,\infty  \right)$ if and only if $h$ is operator concave \cite[Theorem 2.3]{3}.
\end{proof}

\begin{proposition}
Let $f$ be $g-$convex and let $A\in\mathcal{B}(\mathcal{H})$ be self adjoint. If $x\in\mathcal{H}$ is a unit vector, then
$$f\left(\left<Ax,x\right>\right)\leq g^{-1}\left(\left<(g\circ f)(A)x,x\right>\right)\leq \left<f(A)x,x\right>.$$
\end{proposition}
\begin{proof}
Since $f$ is $g-$convex, we have
\begin{align*}
(g\circ f)\left(\left<Ax,x\right>\right)&\leq \left<(g\circ f)(A)x,x\right>\leq g\left(\left<f(A)x,x\right>\right),
\end{align*}
which implies the desired result, upon applying $g^{-1}$ to the above inequalities.
\end{proof}

Let ${{\mathcal{M}}_{n}}$ denote the ${{C}^{*}}$-algebra of $n\times n$ complex matrices with identity $I$ and let ${{\mathcal{H}}_{n}}$ be the set of all Hermitian matrices in ${{\mathcal{M}}_{n}}$. We denote by ${{\mathcal{H}}_{n}}\left( J \right)$ the set of all Hermitian matrices in ${{\mathcal{M}}_{n}}$ whose spectra are contained in an interval $J\subseteq \mathbb{R}$.  The notation $\prec_w$ will be used to denote weak majorization, while $\lambda(A)$ will denote the eigenvalues vector of the Hermitian matrix $A$, arranged in a decreasing order.
\begin{theorem}\label{10}
Let ${{A}_{1}},\ldots ,{{A}_{k}}\in {{\mathcal{H}}_{m}}\left( J \right)$, $f$ be   $g$-convex  on the real interval $J$, and let ${{w}_{1}},\ldots ,{{w}_{k}}$ be positive scalars such that $\sum\nolimits_{i=1}^{k}{{{w}_{i}}}=1$. Then
\[\lambda \left( f\left( \sum\limits_{i=1}^{k}{{{w}_{i}}{{A}_{i}}} \right) \right)~{{\prec }_{w}}~\lambda \left( {{g}^{-1}}\left( \sum\limits_{i=1}^{k}{{{w}_{i}}(g\circ f)\left( {{A}_{i}} \right)} \right) \right)~{{\prec }_{w}}~\lambda \left( \left( \sum\limits_{i=1}^{k}{{{w}_{i}}f\left( {{A}_{i}} \right)} \right) \right).\]
\end{theorem}
\begin{proof}
Let ${{\lambda }_{1}},\ldots ,{{\lambda }_{n}}$ be the eigenvalues of $\sum\nolimits_{i=1}^{k}{{{w}_{i}}{{A}_{i}}}$ and let ${{x}_{1}},\ldots ,{{x}_{n}}$ be the corresponding orthonormal eigenvectors arranged such that $f\left( {{\lambda }_{1}} \right)\ge \ldots \ge f\left( {{\lambda }_{n}} \right)$. Therefore, for $1\le l\le n$,
\[\begin{aligned}
   \sum\limits_{\ell=1}^{l}{{{\lambda }_{\ell}}\left( f\left( \sum\limits_{i=1}^{k}{{{w}_{i}}{{A}_{i}}} \right) \right)}&=\sum\limits_{\ell=1}^{l}{f\left( \left\langle \sum\limits_{i=1}^{k}{{{w}_{i}}{{A}_{i}}{{x}_{\ell}},{{x}_{\ell}}} \right\rangle  \right)} \\ 
 & \le \sum\limits_{\ell=1}^{l}{{{g}^{-1}}\left( \sum\limits_{i=1}^{k}{{{w}_{i}}(g\circ f)\left( \left\langle {{A}_{i}}{{x}_{\ell}},{{x}_{\ell}} \right\rangle  \right)} \right)} \quad \text{(by \eqref{1})}\\ 
 & \le \sum\limits_{\ell=1}^{l}{{{g}^{-1}}\left( \sum\limits_{i=1}^{k}{{{w}_{i}}\left( \left\langle (g\circ f)\left( {{A}_{i}} \right){{x}_{\ell}},{{x}_{\ell}} \right\rangle  \right)} \right)}\\
 & \qquad \text{(since $g\circ f$ is convex and $g^{-1}$ is increasing)}\\ 
 & =\sum\limits_{\ell=1}^{l}{{{g}^{-1}}\left( \left\langle \sum\limits_{i=1}^{k}{{{w}_{i}}(g\circ f)\left( {{A}_{i}} \right)}{{x}_{\ell}},{{x}_{\ell}} \right\rangle  \right)} \\ 
 & \le \sum\limits_{\ell=1}^{l}{\left\langle {{g}^{-1}}\left( \sum\limits_{i=1}^{k}{{{w}_{i}}(g\circ f)\left( {{A}_{i}} \right)} \right){{x}_{\ell}},{{x}_{\ell}} \right\rangle } \\
 &\qquad \text{(since $g^{-1}$ is convex)}\\ 
 & \le \sum\limits_{\ell=1}^{l}{{{\lambda }_{\ell}}\left( {{g}^{-1}}\left( \sum\limits_{i=1}^{k}{{{w}_{i}}(g\circ f)\left( {{A}_{i}} \right)} \right) \right)}.  
\end{aligned}\]
Therefore,
\begin{equation}\label{9}
\lambda \left( f\left( \sum\limits_{i=1}^{k}{{{w}_{i}}{{A}_{i}}} \right) \right)~{{\prec }_{w}}~\lambda \left( {{g}^{-1}}\left( \sum\limits_{i=1}^{k}{{{w}_{i}}(g\circ f)\left( {{A}_{i}} \right)} \right) \right).
\end{equation}
On the other hand, by \cite[Remark 2.1 (ii)]{1}
\begin{equation}\label{8}
\lambda \left( {{g}^{-1}}\left( \sum\limits_{i=1}^{k}{{{w}_{i}}(g\circ f)\left( {{A}_{i}} \right)} \right) \right)~{{\prec }_{w}}~\lambda \left( \left( \sum\limits_{i=1}^{k}{{{w}_{i}}{{g}^{-1}}\left( (g\circ f)\left( {{A}_{i}} \right) \right)} \right) \right)=\lambda \left( \left( \sum\limits_{i=1}^{k}{{{w}_{i}}f\left( {{A}_{i}} \right)} \right) \right).
\end{equation}
Combining \eqref{7} and \eqref{8}, we infer that
\[\lambda \left( f\left( \sum\limits_{i=1}^{k}{{{w}_{i}}{{A}_{i}}} \right) \right)~{{\prec }_{w}}~\lambda \left( {{g}^{-1}}\left( \sum\limits_{i=1}^{k}{{{w}_{i}}(g\circ f)\left( {{A}_{i}} \right)} \right) \right)~{{\prec }_{w}}~\lambda \left( \left( \sum\limits_{i=1}^{k}{{{w}_{i}}f\left( {{A}_{i}} \right)} \right) \right).\]
This completes the proof of the theorem. 
\end{proof}

As a direct consequence of Theorem \ref{10}, we have the following result:
\begin{corollary}\label{11}
Let ${{A}_{1}},\ldots ,{{A}_{k}}\in {{\mathcal{H}}_{m}}\left( J \right)$, and let ${{w}_{1}},\ldots ,{{w}_{k}}$ be positive scalars such that $\sum\nolimits_{i=1}^{k}{{{w}_{i}}}=1$. Then for any $r\ge 2$,
	\[\begin{aligned}
   \lambda \left( {{\left( \sum\limits_{i=1}^{k}{{{w}_{i}}{{A}_{i}}} \right)}^{r}} \right)&~{{\prec }_{w}}~\frac{1}{2}\lambda \left( 2\sum\limits_{i=1}^{k}{{{w}_{i}}\left( A_{i}^{r}+A_{i}^{\frac{r}{2}} \right)}+I-\sqrt{4\sum\limits_{i=1}^{k}{{{w}_{i}}\left( A_{i}^{r}+A_{i}^{\frac{r}{2}} \right)}+I} \right) \\ 
 &~ {{\prec }_{w}}~\lambda \left( \sum\limits_{i=1}^{k}{{{w}_{i}}A_{i}^{r}} \right).  
\end{aligned}\]
\end{corollary}
\begin{proof}
Letting $g\left( x \right)=x+\sqrt{x}$ on $\left[ 0,\infty  \right)$. Then $g'\left( x \right)=\frac{1}{2\sqrt{x}}+1\ge 0$ and $g''\left( x \right)=-\frac{1}{4\sqrt{{{x}^{3}}}}\le 0$. Thus $g$ is increasing and concave. Put $f\left( x \right)={{x}^{r}}\left( r\ge 2 \right)$ on $\left[ 0,\infty  \right)$. Therefore, $g\left( f\left( x \right) \right)={{x}^{r}}+{{x}^{\frac{r}{2}}}$ and $g''\left( f\left( x \right) \right)=\frac{r\left( \left( 4r-4 \right){{x}^{r}}+\left( r-2 \right){{x}^{\frac{r}{2}}} \right)}{4{{x}^{2}}}\ge 0$, namely $g\left( f\left( x \right) \right)$ is a convex function. Since ${{g}^{-1}}\left( x \right)=\frac{2x+1-\sqrt{4x+1}}{2}$, we get the desired result.
\end{proof}

We give an example to clarify the situation in Corollary \ref{11}.
\begin{example}
Letting $k=2$, ${{A}_{1}}=\left[ \begin{matrix}
   2 & -1  \\
   -1 & 1  \\
\end{matrix} \right]$, ${{A}_{2}}=\left[ \begin{matrix}
   2 & 1  \\
   1 & 2  \\
\end{matrix} \right]$, ${{w}_{1}}={{w}_{2}}={1}/{2}\;$, and $r=2$. A simple calculation shows that
\[\lambda \left( {{\left( \frac{{{A}_{1}}+{{A}_{2}}}{2} \right)}^{2}} \right)\approx \left\{ 4,2.2 \right\},\]
\[\frac{1}{2}\lambda \left( A_{1}^{2}+{{A}_{1}}+I-\sqrt{2\left( A_{1}^{2}+{{A}_{1}} \right)+I}+A_{2}^{2}+{{A}_{2}}+I-\sqrt{2\left( A_{2}^{2}+{{A}_{2}} \right)+I} \right)\approx \left\{ 4.5,2.8 \right\},\]
and
\[\lambda \left( \frac{A_{1}^{2}+A_{2}^{2}}{2} \right)\approx \left\{ 5.1,3.3 \right\}\]
that is, we have
\[\left\{ 4,2.2 \right\}~{{\prec }_{w}}~\left\{ 4.5,2.8 \right\}~{{\prec }_{w}}~\left\{ 5.1,3.3 \right\}.\]
\end{example}

Kosem \cite{kosem} proved that if $k:(0,\infty)\to\mathbb{R}$ is a convex (resp. concave) function with $h\left( 0 \right)=0$, then 
	\[\left\| k\left( A \right)+k\left( B \right) \right\|\le \left( \text{resp}\text{. }\ge  \right)\left\| k\left( A+B \right) \right\|,\] for positive matrices $A,B\in\mathcal{M}_n.$ It turns out that $g-$convex functions satisfy better bounds, as follows.

\begin{theorem}
Let $A,B\in \mathcal{M}_n$ be positive  and let $f$ be a $g$-convex function on the interval $[0,\infty),$ with $\left( g\circ f \right)\left( 0 \right)\le 0$ and $g(0)\ge 0$. Then
	\[\left\| f\left( A \right)+f\left( B \right) \right\|\le \left\| {{g}^{-1}}\left( g\circ f\left( A \right)+g\circ f\left( B \right) \right) \right\|\le \left\| f\left( A+B \right) \right\|.\]
\end{theorem}
\begin{proof}
If $f$ is a $g$-convex, we get
	\[\left\| g\circ f\left( A \right)+g\circ f\left( B \right) \right\|\le \left\| g\circ f\left( A+B \right) \right\|.\]
Since $g$ is increasing and concave, we infer that
	\[\left\| g\left( f\left( A \right)+f\left( B \right) \right) \right\|\le \left\| g\circ f\left( A \right)+g\circ f\left( B \right) \right\|\le \left\| g\circ f\left( A+B \right) \right\|.\]
Now, applying ${{g}^{-1}}$, to get
	\[\left\| f\left( A \right)+f\left( B \right) \right\|\le \left\| {{g}^{-1}}\left( g\circ f\left( A \right)+g\circ f\left( B \right) \right) \right\|\le \left\| f\left( A+B \right) \right\|.\]
	This completes the proof of the theorem.
\end{proof}

Related to the index of convexity, we have the following result. The proof is an immediate consequence of Proposition \ref{prop_log_conv_ind}, noting that the function $t\mapsto \|A^{t}XB^{1-t}\|$ is $\log-$convex \cite{saboam}.
\begin{corollary}
Let $A,B\in\mathcal{M}_n$ be positive definite matrices and let $X\in\mathcal{M}_n$. If $\|\cdot\|$ is a unitarily invariant norm on $\mathcal{M}_n$, then the function $f:\mathbb{R}\to [0,\infty)$ defined by $$f(t)=\left\|A^{t}XB^{1-t}\right\|$$ has index of convexity $\infty.$
\end{corollary}

\section{Some applications to entropies}
In this section, we give a new lower bound of quantum relative entropy as an application in this topic.
In quantum information theory \cite{NC2000, OP2004},  the quantum entropy (von Neumann entropy) \cite{Neu}
defined by $S(\rho):=-Tr[\rho\log \rho]$ for a density operator $\rho$, is an important quantity. A density operator is a self adjoint positive operator with unit trace.
The quantum relative entropy \cite{Umegaki} is also important quantity and it is defined by
$$
D(\rho|\sigma):=Tr[\rho(\log \rho-\log \sigma)] 
$$
for two density operators $\rho$ and $\sigma$.  It is known the non-nagativity of quantum relative entropy, $D(\rho|\sigma) \geq 0$. Our lower bound modify this in the following theorem. To show our theorem we give the following lemma.

\begin{lemma}\label{lemma4.1}
\hfill
\begin{itemize}
\item[(i)]
If $f$ is a $g$-convex, then we have
\begin{equation}\label{eq01_lemma4.1}
   f\left( a \right)+f'\left( a \right)\left( b-a \right)\le f\left( a \right)+\left( {{g}^{-1}} \right)'\left( h\left( a \right) \right)\left( h\left( b \right)-h\left( a \right) \right) 
  \le f\left( b \right).  
\end{equation}
\item[(ii)] 
If $f$ is a $g$-concave, then we have
\begin{equation}\label{eq02_lemma4.1}
   f\left( a \right)+f'\left( a \right)\left( b-a \right)\ge f\left( a \right)+\left( {{g}^{-1}} \right)'\left( h\left( a \right) \right)\left( h\left( b \right)-h\left( a \right) \right) 
  \ge f\left( b \right).  
\end{equation}
\end{itemize}
\end{lemma}
\begin{proof}
Since clearly ${{g}^{-1}}$ is increasing convex under the assumptions of lemma, one can check that
\[\begin{aligned}
   f\left( \left( 1-v \right)a+vb \right)&=g^{-1}\circ h\left( a+v\left( b-a \right) \right) \\ 
 & \le {{g}^{-1}}\left( h\left( a \right)+v\left( h\left( b \right)-h\left( a \right) \right) \right) \quad \text{(Convexity of $h$ and $g$ is increasing)}\\ 
 & \le f\left( a \right)+v\left( f\left( b \right)-f\left( a \right) \right)\quad \text{(Convexity of $g$)}.  
\end{aligned}\]
Therefore,
\[\begin{aligned}
   \frac{f\left( a+v\left( b-a \right) \right)-f\left( a \right)}{v}&\le \frac{{{g}^{-1}}\left( h\left( a \right)+v\left( h\left( b \right)-h\left( a \right) \right) \right)-{{g}^{-1}}\left( h\left( a \right) \right)}{v} \\ 
 & \le f\left( b \right)-f\left( a \right)  
\end{aligned}\]
Now, if $v\to 0$, we get \eqref{eq01_lemma4.1}. (ii) can be proven similarly.
\end{proof}

\begin{theorem}
For two density operators $\rho$ and $\sigma$, we have
\begin{equation}\label{comment02_ineq03}
D(\rho|\sigma) \geq S(\sigma)-S(\rho)+Tr[\exp(-\rho\log \rho)\exp(\sigma\log \sigma)-I] \geq 0.
\end{equation}
\end{theorem}

\begin{proof}
We take a concave function $f(t):=-t\log t$ for $0<t \leq 1$ and an increasing convex function $g(t):=\exp(t)$. Then $h(t):=g\circ f(t)=\exp(-t\log t)=t^{-t}$ is concave on $(0,1]$. Since 
$g\circ f''(t)=t^{-t}(1+\log t)^2-t^{-t-1} \leq 0$ for $t \in (0,1]$. To prove $g\circ f''(t) \leq 0$, it is sufficient to consider the function $k(t):=t(1+\log t)^2$ on $0<t \leq 1$. Then we have $k'(t)=(\log t +1)(\log t+3)$. We also easily find that $k'(t) \geq 0$ for $0<t < e^{-3}$, $k'(t)\leq 0$ for $e^{-3}<t<e^{-1}$ and $k'(t)\geq 0$ for $e^{-1}<t \leq 1$. Since $k(e^{-3})=4e^{-3}\simeq 0.199148 < 1=k(1)$, the function $k(t)$ take a maximum value $1$ when $t=1$ for $0<t\leq 1$. Thus we have $k(t) \leq 1$ so that $t(1+\log t)^2 \leq 1$ which proves $g\circ f''(t) \leq 0$.
Thus we have the following inequalities by Lemma \ref{lemma4.1}(ii)
\begin{equation}\label{comment02_ineq02}
   f(x)-f\left( y \right)-f'\left( y \right)\left( x-y \right)\le f(x)-f\left( y \right)-\left( {{g}^{-1}} \right)'\left( h\left( y \right) \right)\left( h\left( x \right)-h\left( y \right) \right) \leq 0
\end{equation}
We take spectral decompositions $\rho = \sum_i \lambda_i P_i$ and $\sigma = \sum_j \mu_j Q_j$ with $\sum_iP_i =\sum_j Q_j =I$.
Then we have the following inequalities:
\begin{eqnarray*}
-Tr[\rho(\log \rho-\log \sigma)]&=&Tr\left[-\rho\log \rho +\sigma \log \sigma -(-\log \sigma -I)(\rho-\sigma)\right]\\
&=&\sum_{i,j}Tr\left[P_i\left\{-\lambda_i\log \lambda_i+\mu_j\log \mu_j-(-\log\mu_j-1)(\lambda_i-\mu_j)\right\}Q_j\right]\\
&=&\sum_{i,j}\left\{-\lambda_i\log \lambda_i+\mu_j\log \mu_j-(-\log\mu_j-1)(\lambda_i-\mu_j)\right\}Tr[P_iQ_j]\\
&\le &\sum_{i,j}\left\{-\lambda_i\log \lambda_i+\mu_j\log \mu_j-(\mu_j^{\mu_j})(\lambda_i^{-\lambda_i}-\mu_j^{-\mu_j})\right\}Tr[P_iQ_j]\\
&=&\sum_{i,j}\left\{-\lambda_i\log \lambda_i+\mu_j\log \mu_j-\mu_j^{\mu_j} \lambda_i^{-\lambda_i}+1 \right\}Tr[P_iQ_j]\\
&\leq& 0.
\end{eqnarray*}
The inequalities above are due to \eqref{comment02_ineq02}.
Finally we  derive
\begin{eqnarray*}
&&\sum_{i,j}\left\{-\lambda_i\log \lambda_i+\mu_j\log \mu_j-\mu_j^{\mu_j} \lambda_i^{-\lambda_i}+1 \right\}Tr[P_iQ_j]\\
&&=\sum_{i,j}Tr\left[P_i\left\{-\lambda_i\log \lambda_i+\mu_j\log \mu_j-\exp(\mu_j \log \mu_j ) \exp (-\lambda_i\log \lambda_i)+1 \right\}Q_j\right]\\
&&=Tr[-\rho \log \rho +\sigma\log \sigma -\exp(\sigma\log\sigma)\exp(-\rho\log \rho) +I],
\end{eqnarray*}
since we have $\sum_{i,j}Tr[P_if(\lambda_i)g(\mu_j)Q_j]=Tr[\sum_if(\lambda_i)P_i\sum_jg(\mu_j)Q_j]=Tr[f(\rho)g(\sigma)]$.
Thus we have
$$
-Tr[\rho(\log \rho-\log \sigma)] \leq Tr[-\rho \log \rho +\sigma\log \sigma -\exp(-\rho\log \rho)\exp(\sigma\log\sigma) +I] \leq 0,
$$
which implies \eqref{comment02_ineq03}.
\end{proof}

\begin{remark}
The inequalities \eqref{comment02_ineq03} are equivalent to
$$
Tr[(\sigma-\rho)\log \sigma] \geq Tr[\exp(-\rho\log \rho)\exp(\sigma\log \sigma)-I] \geq S(\rho)-S(\sigma).
$$
If we consider the special case $\rho=\sigma$, then both sides in the above inequalities become to $0$, so that equality holds. 
\end{remark}

\begin{remark}
From \eqref{comment02_ineq03}, we have the lower bound of quantum Jeffrey divergence \cite{FYK}:
$$
J(\rho|\sigma):=\frac{1}{2}\left(D(\rho|\sigma)+D(\sigma|\rho)\right)
$$
as
$$
J(\rho|\sigma) \geq \frac{1}{2}\left(Tr[\exp(-\rho\log \rho)\exp(\sigma\log \sigma)+\exp(\rho\log \rho)\exp(-\sigma\log \sigma)-2I]\right).
$$
\end{remark}

The following examples show that the inequalities \eqref{comment02_ineq03} can be strict.
\begin{example}
We take density matrices as
\[\rho : = \frac{1}{7}\left[ {\begin{array}{*{20}{c}}
2&2\\
2&5
\end{array}} \right],\quad\sigma : = \frac{1}{6}\left[ {\begin{array}{*{20}{c}}
3&1\\
1&3
\end{array}} \right].\]
Then we have
$$
D(\rho|\sigma) \simeq 0.14388
$$
and
$$
S(\sigma)-S(\rho)+Tr[\exp(-\rho\log \rho)\exp(\sigma\log \sigma)-I] \simeq 0.0141518.
$$
For the case 
\[\rho : =\frac{1}{6}\left[ {\begin{array}{*{20}{c}}
3&1\\
1&3
\end{array}} \right],\quad\sigma : =  \frac{1}{7}\left[ {\begin{array}{*{20}{c}}
2&2\\
2&5
\end{array}} \right]\]
we also have
$$
D(\rho|\sigma) \simeq 0.174615
$$
and
$$
S(\sigma)-S(\rho)+Tr[\exp(-\rho\log \rho)\exp(\sigma\log \sigma)-I] \simeq 0.0155788.
$$
\end{example}

%

\section*{Acknowledgement}
The author (S.F.) was partially supported by JSPS KAKENHI Grant Number 16K05257.

\vskip 0.3 true cm

{\tiny (M. Sababheh)  Department of Basic Sciences, Princess Sumaya University For Technology, Al Jubaiha, Amman 11941, Jordan.}

{\tiny \textit{E-mail address:} sababheh@psut.edu.jo}

{\tiny \vskip 0.3 true cm }

{\tiny (S. Furuichi) Department of Information Science, College of Humanities and Sciences, Nihon University,
3-25-40, Sakurajyousui, Setagaya-ku, Tokyo, 156-8550, Japan}

{\tiny \textit{E-mail address:} furuichi@chs.nihon-u.ac.jp}

{\tiny \vskip 0.3 true cm }

{\tiny (H. R. Moradi) Department of Mathematics, Payame Noor University (PNU), P.O. Box 19395-4697, Tehran, Iran.}

{\tiny \textit{E-mail address:} hrmoradi@mshdiau.ac.ir }

\begin{thebibliography}{9}
\bibitem{3}
T. Ando and F. Hiai, {\it Operator log-convex functions and operator means}, Math. Ann., {\bf350}(3) (2011), 611--630.

\bibitem{FYK}
S. Furuichi, K. Yanagi and K. Kuriyama, {\it On bounds for symmetric divergence measures}, AIP Conf. Proc., {\bf 1853}(2017), 080002.

\bibitem{4}
F. Hansen, {\it An operator inequality}, Math. Ann., {\bf246} (1980), 249--250.


\bibitem{kosem}
T. Kosem, {\it Inequalities between $\left\| f\left( A+B \right) \right\|$ and $\left\| f\left( A \right)+f\left( B \right) \right\|$}, Linear Algebra Appl., {\bf418} (2006), 153--160.
 
\bibitem{1}
H. R. Moradi and M. Sababheh, {\it Eigenvalue inequalities for $n$-tuple of matrices }, Linear Multilinear Algebra. https://doi.org/10.1080/03081087.2019.1664384


\bibitem{NC2000} 
M. A. Nielsen and I. L. Chuang, {\it Quantum computation and quantum information}, Cambridge University Press, 2000. 

\bibitem{OP2004}
M. Ohya and D. Petz, {\it Quantum entropy and its use}, Springer-Verlag, Second Edition 2004.

\bibitem{2}
H. L. Pedersen and M. Uchiyama, {\it Inverses of operator convex functions, in Ordered structures and applications}, Papers from Positivity VII, ed. M. de Jeu, B. de Pagter, O. van Gaans and
M. Veraar, pp 363--370, Birkh\"auser/Springer, 2016.

\bibitem{Petz}
 D. Petz, {\it Quantum information theory and quantum statistics}, Springer, 2004.

\bibitem{sab_bull}
M. Sababheh, H. R. Moradi and S. Furuichi, {\it Integrals refining convex inequalities}, Bull. Malays. Math. Sci. Soc. (2019). https://doi.org/10.1007/s40840-019-00839-0.

\bibitem{saboam}
M. Sababheh, {\it Log and Harmonically log-convex functions related to matrix norms}, Oper. Matrices., {\bf 10}(2) (2016), 453--465.

\bibitem{Umegaki} 
H. Umegaki, {\it Conditional expectation in an operator algebra, IV (entropy and information)}, Kodai Math. Sem. Rep., {\bf 14} (1962), 59--85.

\bibitem{Neu} 
J. von Neumann, {\it Thermodynamik quantenmechanischer Gesamtheiten}, 
Nachr. Ges Wiss. G\"ottingen, (1927), 273--291.
\end{thebibliography}
\end{document}